\documentclass[12pt,reqno]{amsart}
\usepackage{amsmath, amssymb, latexsym, amscd, amsthm,amsfonts,amstext}
\usepackage[mathscr]{eucal}
\theoremstyle{plain}
\newtheorem{lemma}{Lemma}[section]
\newtheorem{theorem}[lemma]{Theorem}

\newtheorem{proposition}[lemma]{Proposition}

\newtheorem{remark}[lemma]{Remark}

\def\underset#1#2{{\mathrel{\mathop {{}_{} {#2}}\limits_{{#1}_{}}}}}
\def\upplim_#1{\underset{#1}{\overline\lim}\;}
\def\lowlim_#1{\underset{#1}{\underline\lim}\;}
\newcommand{\ord}{{\mathrm{ord}}}

\numberwithin{equation}{section}

\begin{document}
\title[ ]{A generalization of effective schmidt's subspace theorem for projective varieties over function fields}
\author{Giang Le}
\setlength{\baselineskip}{16pt}
\maketitle

\begin{abstract}{We establish an effective version of Schmidt's subspace theorem on a smooth projective variety $\mathcal{X}$ over function fields of characteristic zero for hypersurfaces located in $m-$subgeneral position with respect to $\mathcal{X}$. Our result generalizes and improves the results of An-Wang \cite{AW}, Ru-Wang \cite{RW} and the author \cite{G4}.}
\end{abstract}

\def\thefootnote{\empty}
\setlength{\baselineskip}{16pt}
\maketitle
\section{Introduction}
In \cite{AW}, An and Wang obtained an effective Schmidt's subspace theorem for non-linear forms over function fields. In \cite{RW}, Ru and Wang extended such effective results to divisors of a projective variety $\mathcal{X}\subset\mathbb{P}^M$ coming from hypersurfaces in $\mathbb{P}^M$ in general position with respect to $\mathcal{X}$. In \cite{G4}, we extended Ru-Wang's result to the case of hypersurfaces located in subgeneral position. However, due to the weakness of estimate on lower bound of  Chow weight in \cite[Proposition 2.1]{G4}, our result \cite{G4} does not imply the results of  An-Wang \cite{AW} and Ru-Wang \cite{RW}. In \cite[Theorem 1.4]{Q}, S.D.Quang gave a better estimate on lower bound of Chow weight. It helps us to improve our result in \cite{G4}. This new result improves and generalize all the existing results in such the direction.

Here, let $\mathcal{X}$ be a $n$-dimensional projective subvariety of $\mathbb{P}^M$ defined over $K$ and $m, q$ be positive integers with $m\geq n$ and $q\geq m+1$. Recall that homogeneous polynomials $Q_1,\ldots,Q_q\in K[X_0,\ldots,X_M]$ are said to be in \emph{m-subgeneral position with respect to $\mathcal{X}$} if $\cap_{j=1}^{m+1}(\{Q_{i_j}=0\})\cap \mathcal{X}(\bar{K})=\emptyset$ for any distinct $i_1,\ldots,i_{m+1}\in \{1,\ldots,q\},$ where $\bar{K}$ is the algebraic closure of $K$. When $m=n,$ they are said to be  in \emph{general position with respect to $\mathcal{X}$}.


To state our results, we will recall some definitions and basic facts from algebraic geometry.

Let $k$ be an algebraically closed field of characteristic 0 and let $V$ be  a projective variety (always assumed irreducible), non-singular in codimension 1 and defined over $k$. For the rest of paper, we shall fix an embedding of  $V$ such that $V\subset \mathbb{P}^{M_0}$ for some positive integer $M_0$. 

Denote by $K=k(V)$ the function field of $V$. Let $M_K$ be the set of discrete absolute values of the function field $K$ obtained from the prime divisors of $V$.  Let $\mathfrak{p}$ be a prime divisor of $V$ over $k$. Such a prime divisor determines its local ring in the function field $k(V)$ and this local ring is a discrete valuation ring. Thus, we have the notion of order  at $\mathfrak{p}$ of a function $x\in K, x\not=0$, noted $\ord_\mathfrak{p}x$. We can associate to \emph{x} its divisors
==$$(x)=\sum_{\mathfrak{p}\in M_K}\ord_\mathfrak{p}(x)\mathfrak{p}.$$
 By the $degree$ of $\mathfrak{p}$, noted deg $\mathfrak{p}$, we shall mean the projective degree, i.e. the number of points of intersection with a generic linear variety of complementary dimension in the given projective embedding. Then we have the sum formula
$$\text{deg}(x)=\sum_{\mathfrak{p}\in M_K}\ord_\mathfrak{p}(x)\deg\,\mathfrak{p}=0$$ 
for all $x\in K^*$.

Let ${\bf x}=[x_0:x_1:\cdots:x_M]\in\mathbb{P}^M(K)$ and define
$$e_{\mathfrak{p}}({\bf x}):=\min\limits_{0\leq i\leq M}\{\ord_{\mathfrak{p}}(x_i)\}.$$
We define the (logarithmic) height of ${\bf x}$ by:
$$h({\bf x})=-\sum_{\mathfrak{p}\in M_K}e_{\mathfrak{p}}({\bf x}) \deg\,\mathfrak{p}.$$
By the sum formula, the height function is well-defined on $\mathbb{P}^M(K)$.

Let $Q=\sum_Ia_I {\bf x}^I$ be a homogeneous polynomial of degree $d$ in $K[X_0,\ldots, X_M],$ where ${\bf x}^I=x_0^{i_0}\ldots x_M^{i_M}$ and the sum is taken over all index sets $I=\{i_0,\ldots,i_M\}$ such that $i_j\geq 0$ and $\sum\limits_{j=0}^Mi_j=d$. For  each $\mathfrak{p}\in M_K,$ we set
$$e_{\mathfrak{p}}(Q):=\text{min}_I\{\ord_{\mathfrak{p}}(a_I)\}.$$
The height of a homogeneous polynomial $Q$ of degree $d$ in $K[X_0,\ldots, X_M]$ is defined by the height of coefficients:
$$h(Q)=\sum_{\mathfrak{p}\in M_K}-e_{\mathfrak{p}}(Q)\deg \mathfrak{p}.$$
From the sum formula, we have $h(\alpha Q)=h(Q)$ for all $\alpha\in K^*.$  Since we may assume that one of the non-zero coefficient of $Q$ is $1$, it follows that $h(Q)\geq 0$.

The Weil function $\lambda_{\mathfrak{p},Q}$ is defined by
$$\lambda_{\mathfrak{p},Q}({\bf x}):=\left(\ord_{\mathfrak{p}}(Q({\bf x}))-de_{\mathfrak{p}}({\bf x})-e_{\mathfrak{p}}(Q)\right)\text{deg}\,\mathfrak{p}\geq 0$$
for ${\bf x}\in\mathbb{P}^M(K)\backslash\{Q=0\}.$

Let $Q_1, Q_2,\ldots, Q_q$ be homogeneous polynomials of degree $d$ in $K[X_0,\ldots,X_M]$. We define $$e_\mathfrak{p}(Q_1,\ldots,Q_q)=\min\{e_\mathfrak{p}(Q_1),\ldots,e_\mathfrak{p}(Q_q)\}$$
and
$$h(Q_1,\ldots,Q_q)=-\sum_{\mathfrak{p}\in M_K}e_\mathfrak{p}(Q_1,\ldots,Q_q)\deg \mathfrak{p}.$$

Let $\mathcal{X}$ be a $n$-dimensional projective subvariety of $\mathbb{P}^M$ defined over $K$. The height of $\mathcal{X}$ is defined by $$h(\mathcal{X}):=h(F_\mathcal{X}),$$
where $F_\mathcal{X}$ is the Chow form of $\mathcal{X}$.

In this paper, we will prove the following effective version of the generalized Schmidt's  subspace theorem over $K$.\\ 
\noindent
{\bf Main Theorem.} {\it 
 Let $K$ be the function field of a nonsingular projective variety $V$ defined over an algebraically closed field of characteristic 0 and let S be a finite set of prime divisors of V. Let $\mathcal{X}$ be a smooth n-dimensional projective subvariety of $\mathbb{P}^N$ defined over K with projective degree $\triangle_\mathcal{X}$. Let $m, q$ be integers with $m\geq n$ and $q\geq m+1.$ For all $i=1,\ldots, q$, let $Q_i $ be homogeneous polynomials of degree $d_i$ in $K[X_0,\ldots,X_N]$ in m-subgeneral position with respect to $\mathcal{X}$. Then for any given $\epsilon>0$, there exists an effectively computable finite union $\mathcal{W}_{\epsilon}$ of proper algebraic subsets of $\mathbb{P}^N(K)$ not containing $\mathcal{X}$  and effectively computable constants $C_{\epsilon}, C'_{\epsilon}$ such that for any ${\bf x}\in \mathcal{X}(K) \backslash\mathcal{W}_{\epsilon}$ either
$$h({\bf x})\leq C_{\epsilon}$$
or 
$$\sum_{i=1}^q\sum_{\mathfrak{p}\in S}d_i^{-1}\lambda_{\mathfrak{p},Q_i}({\bf x})\leq ((m-n+1)(n+1)+\epsilon)h({\bf x})+C'_{\epsilon}.$$

The algebraic subsets in $\mathcal{W}_\epsilon$ and the constants $C_{\epsilon}, C'_{\epsilon}$ depend on $\epsilon$ and $ N, q, m, K, S, \mathcal{X}$ and the $Q_i.$ Furthermore, the degrees of the algebraic subsets in $\mathcal{W}_\epsilon$ can be bounded above by
$$2(2n+1)d^{n+1}\bigtriangleup_\mathcal{X}\left(\binom{d+N}{N}+q+1\right)\epsilon^{-1}+d,$$
where $d=lcm(d_1,\ldots,d_q)$.}
\begin{remark} The constants $C_\epsilon, C'_\epsilon$ will be given in \eqref{e:c} and \eqref{e:c'}. They may depend on $\epsilon,$ the degree of the canonical divisor class of $V$, the projective degree of $V$, the degree of $S$ (i.e $\sum_{\mathfrak{p}\in S} \deg \mathfrak{p}$), the
projective degree of $\mathcal{X}$, the dimension of $\mathcal{X}$, the height of $\mathcal{X}$ and the $Q_i,$ $q$ and $m, N$.  
\end{remark}
\section{Chow forms, Chow weight of a projective variety}
  Let $\mathcal{Y}$ be a $n$-dimensional  projective subvariety of  $\mathbb{P}^M$ defined over $K$ of degree $\triangle_\mathcal{Y}$. To $\mathcal{Y}$, we can associate, up to a constant scalar, a unique polynomial
$$F_\mathcal{Y}({\bf u}_0,\ldots,{\bf u}_n)=F_\mathcal{Y}(u_{00},\ldots,u_{0M};\ldots; u_{n0},\ldots,u_{nM})$$
in $(n+1)$ blocks of variables ${\bf u}_i=(u_{i0},\ldots,u_{iM}), i=0,\ldots, n,$ which is called the \emph{Chow form} of $\mathcal{Y}$, with the following properties:

$F_\mathcal{Y}$ is irreducible,

 $F_\mathcal{Y}$ is homogeneous in each block $u_i, i=0,\ldots,n,$

$F_\mathcal{Y}({\bf u}_0,\ldots,{\bf u}_n)=0$ if and only if $\mathcal{Y}\cap H_{{\bf u}_0}\cap\ldots\cap H_{{\bf u}_n}$ contains a $\bar{K}$
-rational point,
where $H_{\bf{u}_i}, i=0,\ldots, n$ are hyperplanes given by ${\bf u}_i\cdot{\bf x}=u_{i0}x_0+\cdots+u_{iM}x_M=0.$ It is well-known that the degree of $F_\mathcal{Y}$ in each block ${\bf u}_i$ is $\triangle_\mathcal{Y}$.

Let ${\bf c}=(c_0,\ldots,c_M)$ be a tuple of reals. Let $t$ be an auxiliary variable. We consider the decomposition
\begin{align*}F_\mathcal{Y}(t^{c_0}u_{00},&\ldots, t^{c_M}u_{0M},\ldots,t^{c_0}u_{n0},\ldots,t^{c_M}u_{nM})\\
&=t^{e_0}G_0({\bf u}_0,\ldots,{\bf u}_n)+\cdots+t^{e_r}G_r({\bf u}_0,\ldots,{\bf u}_n),
\end{align*}
with $G_0,\ldots, G_r\in K[u_{00},\ldots, u_{0M};\ldots; u_{n0},\ldots,u_{nM}]$ and $e_0>\ldots>e_r.$ Now, we define the \emph{Chow weight of $\mathcal{Y}$ with respect to ${\bf c}$} by
$$e_\mathcal{Y}({\bf c}):=e_0.$$
\noindent

\noindent
We recall the following estimate on Chow weight of a projective variety $\mathcal{Y}$, due  to S.D. Quang \cite[Theorem 1.4]{Q}.

\begin{proposition}\label{p:41} Let $\mathcal{Y}$ be a n-dimensional  projective algebraic subvariety of $\mathbb{P}^M$ defined over $K$ of degree $\triangle_\mathcal{Y}.$ Let ${\bf c}=(c_0,\ldots,c_M)\in\mathbb{R}^{M+1}_+$. Let $m$ be a positive integer such that $m\geq n.$ Assume that $\mathcal{Y}\cap H_{i_0}\ldots \cap H_{i_m}=\emptyset $. Then,
$$e_\mathcal{Y}({\bf c})\geq \frac{\triangle_\mathcal{Y}}{m-n+1}\cdot (c_{i_0}+\cdots+c_{i_m}).$$
\end{proposition}

\noindent
 We now recall an estimate on heights of Chow forms, due to Ru-Wang \cite[ Lemma 8]{RW}.
\begin{lemma}\label{l:21} Let $\mathcal{X}$ be a projective variety of $\mathbb{P}^M$ defined over $K$ with dimension $n\geq 1$ and degree $\triangle_\mathcal{X}$. Let $\psi: \mathcal{X}\longrightarrow \mathbb{P}^R$ be a finite morphism given by $\psi({\bf x})=[g_0({\bf x}):\cdots: g_R({\bf x})]$, where $g_0,\ldots,g_R$ are homogeneous polynomials of degree $d$ in $K[X_0,\ldots,X_M]$. Let $\mathcal{Y}=\psi(\mathcal{X})$. Then,
$$h(F_\mathcal{Y})\leq d^{n+1}h(F_\mathcal{X})+ (n+1) d^n \triangle_\mathcal{X}h(g_0,\ldots,g_R).$$
\end{lemma}


\section{Proof of  Main Theorem} 

To prove the Main theorem, we need the following lemma.

 \begin{lemma}{\label{l:37}Let $\mathcal{X}$ be a smooth $n$-dimensional  projective subvariety of  $\mathbb{P}^M$ defined over $K$ of degree $\triangle_\mathcal{X}$. 
 		Let $m, q$ be integers with $m\geq n$ and $q\geq m+1$. Let $Q_1,\ldots,Q_q $ be homogeneous polynomials in $K[X_0,\ldots,X_M]$ of degree $d$, in $m$-subgeneral position with respect to $\mathcal{X}$. For  given $\mathfrak{p}\in M_K$, and ${\bf x}\in \mathcal{X}\backslash\cup_{i=1}^q\{Q_i=0\},$ we assume that 
 		\begin{align}\label{e:104}
 		\ord_{\mathfrak{p}}(Q_1({\bf x}))\geq\cdots\geq \ord_{\mathfrak{p}}(Q_q({\bf x})).
 		\end{align}
 		Then
 		\begin{align*}&\ord_{\mathfrak{p}}(Q_i({\bf x}))\deg \mathfrak{p}-d\cdot e_\mathfrak{p}({\bf x})\deg \mathfrak{p}\\
 		&\leq\left(6\max\{(m+1)\triangle_\mathcal{X}, d\}\right)^{(n+1)(M^2+M)}\left(h(F_\mathcal{X})+h(Q_1,\ldots,Q_q)\right)
 		\end{align*}
 		for $\mathfrak{p}\in M_K$ and $m+1\leq i\leq q.$}
 \end{lemma}
 We refer the reader to \cite[Lemma 4.2]{G4} for the proof of the above lemma.

We now recall the following theorem, due to Ru-Wang \cite[Theorem 23]{RW}.
\begin{theorem}[Ru-Wang \cite{RW}]\label{t:41}
Let $K$ be the function field of a nonsingular projective variety $V$ defined over an algebraically closed field of characteristic 0.  Let S be a finite set of prime divisors of V. Let $\mathcal{Y}$ be an n-dimensional smooth projective subvariety of $\mathbb{P}^M$ defined over K. For every ${\bf\mathfrak{p}}\in  M_K$ and ${\bf y}=[y_0:\cdots:y_M]$, we let $c_{\mathfrak{p},i}({\bf y})=(\ord_\mathfrak{p}(y_i)-e_\mathfrak{p}({\bf y}))\cdot\deg\mathfrak{p}\; (0\leq i\leq M)$ and ${\bf c}_\mathfrak{p}({\bf y})=(c_{\mathfrak{p},0}({\bf y}),\ldots,c_{\mathfrak{p},M}({\bf y})).$
Then for a given $\epsilon>0$, there exists an effectively computable finite union $\mathcal{Z}_{\epsilon}$ of proper algebraic subsets of $\mathbb{P}^M(K)$ not containing $\mathcal{Y}$  and effectively computable constants $a_{\epsilon}, a'_{\epsilon}$ such that for any ${\bf y}\in \mathcal{Y}(K) \backslash\mathcal{Z}_{\epsilon}$ either
$$h({\bf y})\leq a_{\epsilon}(h(F_\mathcal{Y})+1)$$
or 
$$\sum_{\mathfrak{p}\in S}{ e}_\mathcal{Y}({\bf c}_\mathfrak{p}({\bf y}))\leq (n+1+\epsilon)\bigtriangleup_\mathcal{Y}\cdot h({\bf y})+a'_{\epsilon}(h(F_\mathcal{Y})+1).$$

The algebraic subsets in $\mathcal{Z}_\epsilon$ and the constants $a_{\epsilon}, a'_{\epsilon}$  depend on $\epsilon$ and $ M, K, S$ and $\mathcal{Y}$. Furthermore, the degrees of the algebraic subsets in $\mathcal{Z}_\epsilon$ can be bounded above by $1+2(2n+1)\bigtriangleup_\mathcal{Y}(M+1)\epsilon^{-1}.$
\end{theorem}

We first use Theorem \ref{t:41} and Proposition \ref{p:41} to prove Theorem \ref{t:42}. Then, we will show that the main theorem is an implication of  Theorem \ref{t:42}.
 \begin{theorem}\label{t:42} Let $K$ be the function field of a nonsingular projective variety $V$ defined over an algebraically closed field of characteristic 0.  Let S be a finite set of prime divisors of V. Let $m, n$ be positive integers such t==hat $m\geq n.$ Let $\mathcal{Y}$ be an $n$-dimensional smooth projective subvariety of $\mathbb{P}^M$ defined over $K$.  Denote by $I_0$ the subset of $\{0,\ldots, M\}$ such that for all distinct $i_0,\ldots, i_m\in I$, we have $\mathcal{Y}\cap H_{i_0}\cdots\cap H_{i_m}=\emptyset$ .
 Let $\epsilon>0$. Then  there exists an effectively computable finite union $\mathcal{R}_{\epsilon}$ of proper algebraic subsets of $\mathbb{P}^M(K)$ not containing $\mathcal{Y}$  and effectively computable constants $b_{\epsilon}, b'_{\epsilon}$ such that for any ${\bf y}\in \mathcal{Y} \backslash\mathcal{R}_{\epsilon}$ either
$$h({\bf y})\leq b_{\epsilon}(h(F_\mathcal{Y})+1)$$
or 
$$\sum_{\mathfrak{p}\in S}\max_I\sum_{i\in I}\lambda_{\mathfrak{p},Y_i}({\bf y})\leq ((m-n+1)(n+1)+\epsilon)\bigtriangleup_\mathcal{Y}\cdot h({\bf y})+b'_{\epsilon}(h(F_\mathcal{Y})+1).$$
Here the maximum is taken over all subsets $I$ of $I_0$ with cardinality $m+1$.

The algebraic subsets in $\mathcal{R}_\epsilon$ and the constants $b_{\epsilon}, b'_{\epsilon}$  depend on $\epsilon$ and $ M, K, S$ and $\mathcal{Y}$. Furthermore, the degrees of the algebraic subsets in $\mathcal{R}_\epsilon$ can be bounded above by $1+2(2n+1)\bigtriangleup_\mathcal{Y}(M+1)\epsilon^{-1}.$
\end{theorem}
  
\begin{proof} 

For every ${\bf\mathfrak{p}}\in  M_K$ and ${\bf y}=[y_0:\cdots:y_M]$, we let 
$$c_{\mathfrak{p},i}({\bf y})=(\ord_\mathfrak{p}(y_i)-e_\mathfrak{p}({\bf y}))\cdot\deg\mathfrak{p}\qquad (0\leq i\leq M)$$ and ${\bf c}_\mathfrak{p}({\bf y})=(c_{\mathfrak{p},0}({\bf y}),\ldots,c_{\mathfrak{p},M}({\bf y})).$

  Theorem \ref{t:41} implies that  for a given $\epsilon>0$, there exists an effectively computable finite union $\mathcal{Z}_{\epsilon}$ of proper algebraic subsets of $\mathbb{P}^M(K)$ not containing $\mathcal{Y}$  and effectively computable constants $a_{\epsilon/m-n+1}, a'_{\epsilon/m-n+1}$ such that for any ${\bf y}\in \mathcal{Y}(K) \backslash\mathcal{Z}_{\epsilon}$ either
$$h({\bf y})\leq a_{\epsilon/m-n+1}(h(F_\mathcal{Y})+1)$$
or 
\begin{align}\label{e:51}
\sum_{\mathfrak{p}\in S}{ e}_\mathcal{Y}({\bf c}_\mathfrak{p}({\bf y}))\leq \left(n+1+\frac{\epsilon}{m-n+1}\right)\bigtriangleup_\mathcal{Y}\cdot h({\bf y})+a'_{\epsilon/m-n+1}(h(F_\mathcal{Y})+1).
\end{align}

Let $I$ be an arbitrary subset of $I_0$ with cardinality $m+1$. 
It follows from Proposition \ref{p:41} that
\begin{align}
\sum_{i\in I}c_{\mathfrak{p},i}({\bf y})\leq \dfrac
{m-n+1}{\triangle_\mathcal{Y}}e_\mathcal{Y}({\bf c}_\mathfrak{p}({\bf y})).
\end{align}
On the other hand, by the definition, $\lambda_{\mathfrak{p},Y_i}({\bf y})=c_{\mathfrak{p},i}({\bf y}).$ Hence,
\begin{align}\label{e:52}
\sum_{i\in I}\lambda_{\mathfrak{p},Y_i}({\bf y})\leq \dfrac{m-n+1}{\triangle_\mathcal{Y}}e_\mathcal{Y}({\bf c}_\mathfrak{p}({\bf y})).
\end{align}
Therefore,
\begin{align}\label{e:53}
\sum_{\mathfrak{p}\in S}\max_I \sum_{i\in I}\lambda_{\mathfrak{p},Y_i}({\bf y})\leq \dfrac{m-n+1}{\triangle_\mathcal{Y}}\sum_{\mathfrak{p}\in S}e_\mathcal{Y}({\bf c}_\mathfrak{p}({\bf y})).
\end{align}
Set $\mathcal{R}_\epsilon=\mathcal{Z}_\epsilon.$ By combining \eqref{e:51} and \eqref{e:53}, we have
\begin{align*}
\sum_{\mathfrak{p}\in S}\max_I \sum_{i\in I}\lambda_{\mathfrak{p},Y_i}({\bf y})\leq ((m-n+1)(n+1)+\epsilon)\cdot h({\bf y})+a'_{\epsilon/m-n+1}\frac{m-n+1}{\triangle_\mathcal{Y}}(h(F_\mathcal{Y})+1).
\end{align*}
for all ${\bf y}\in \mathcal{Y}\backslash \mathcal{R}_\epsilon.$

The constants $b_\epsilon$ and $b'_\epsilon$ in the assertion can be given by
\begin{align}\label{e:b}
b_\epsilon=a_{\epsilon/m-n+1}; b'_\epsilon=(m-n+1) \cdot a'_{\epsilon/m-n+1},
\end{align}
where $a_\epsilon$ and $a'_\epsilon$ are constants from Theorem \ref{t:41}.
This completes the proof of Theorem \ref{t:42}.
\end{proof}

Now, we will show that Theorem \ref{t:42} implies the main theorem.

{\it Proof of the main theorem.}

Let $d$ is the l.c.m of $d_i', 1\leq i\leq q,$ and let $M_0,\ldots, M_{N_1}$ be all the monomials in $X_0,\ldots, X_N$ of degree $d$. We define the map
  \begin{align}\label{e:54}
  \psi: \mathcal{X}\longrightarrow\mathbb{P}^{N_1+q}, \psi({\bf x})=[M_0({\bf x}):\cdots: M_{N_1}({\bf x}):Q_1^{d/d_1}({\bf x}):\cdots:Q_q^{d/d_q}({\bf x})].
  \end{align}
  Let $\mathcal{Y}=\psi(\mathcal{X})$. Then this map is an embedding and $\mathcal{Y}$ is a smooth projective subvariety of $\mathbb{P}^{N_1+q}$ defined over $K$ with $\dim \mathcal{Y}=n$ and $\deg \mathcal{Y}=:\triangle_\mathcal{Y}\leq d^n\triangle_\mathcal{X}.$ It follows from Lemma \ref{l:21} that
  \begin{align*}
  h(F_\mathcal{Y})\leq d^{n+1}h(F_\mathcal{X})+(n+1)\triangle_\mathcal{X}d^nh(Q_1^{d/d_1},\ldots,Q_q^{d/d_q}).
  \end{align*}
  Since $\mathcal{X}$ is located in m-subgeneral position with respect to $Q_1,\ldots, Q_q,$ then we have $\mathcal{Y}\cap H_{i_0}\ldots \cap H_{i_m}=\emptyset $, for all $i_0,\ldots, i_m\in I_0:=\{N_1+1,\ldots, N_1+q\}.$  We apply Theorem \ref{t:42} to $\mathcal{Y}\in \mathbb{P}^{N_1+q}$ and $I_0:=\{N_1+1,\ldots, N_1+q\}$. 
  Then, for a given $\epsilon >0,$ there exists an effectively computable finite union $\mathcal{R}_{\epsilon}$ of proper algebraic subsets of $\mathbb{P}^{N_1+q}(K)$ not containing $\mathcal{Y}$  and effectively computable constants $b_{\epsilon}, b'_{\epsilon}$ such that for any $\psi({\bf x})\in \mathcal{Y} \backslash\mathcal{R}_{\epsilon}$ either
\begin{align}\label{e:55}
 h(\psi({\bf x}))&\leq b_{\epsilon}(h(F_\mathcal{Y})+1)\notag\\
&\leq b_{\epsilon}(1+d^{n+1}h(F_\mathcal{X})+(n+1)\triangle_\mathcal{X}d^nh(Q_1^{d/d_1},\ldots,Q_q^{d/d_q}))\notag\\
&\leq \tilde{b}_\epsilon(h(F_\mathcal{X})+1)
\end{align}
or 
\begin{align}\label{e:56}
\sum_{\mathfrak{p}\in S}\max_I\sum_{i\in I}\lambda_{\mathfrak{p},Y_i}({\bf y})&\leq ((m-n+1)(n+1)+\epsilon)h(\psi({\bf x}))+b'_{\epsilon}(h(F_\mathcal{Y})+1)\notag\\
&\leq ((m-n+1)(n+1)+\epsilon)h(\psi({\bf x}))+\tilde{b}^{'}_{\epsilon}(h(F_\mathcal{X})+1).
\end{align}
Here the maximum is taken over all subsets $I$ of $\{N_1+1,\ldots, N_1+q\}$ with cardinality $m+1$ and the constants $\tilde{b}_\epsilon$ and $\tilde{b}'_\epsilon$ are given by
\begin{align}\label{e:57}
\tilde{b}_\epsilon=b_\epsilon\cdot (d^{n+1}+(n+1)\triangle_\mathcal{X}d^nh(Q_1^{d/d_1},\ldots,Q_q^{d/d_q}))
\end{align}
and
\begin{align}\label{e:58}
\tilde{b}'_\epsilon=b'_\epsilon\cdot (d^{n+1}+(n+1)\triangle_\mathcal{X}d^nh(Q_1^{d/d_1},\ldots,Q_q^{d/d_q})).
\end{align}
Here $b_\epsilon$ and $b'_\epsilon$ are the constants from Theorem \ref{t:42} with $M=N_1+q$. Notice that the degrees of the algebraic subsets in $\mathcal{R}_\epsilon$ can be bounded by
\begin{align}\label{e:59}
2(2n+1)\bigtriangleup_\mathcal{Y}(M+1)\epsilon^{-1}+1\leq 2(2n+1)d^n\bigtriangleup_\mathcal{X}\left(\binom{d+N}{N}+q+1\right)\epsilon^{-1}+1.
\end{align}
For a given ${\bf x}\in\mathcal{X}\backslash \cup_{i=1}^q\{Q_i=0\}$ and a fixed $\mathfrak{p}\in S$, we may reindex the $Q_i$ so that $(d/d_1)\ord_\mathfrak{p}(Q_1({\bf x}))\geq\cdots\geq (d/d_q)\ord_\mathfrak{p}(Q_q({\bf x})).$  Since $Q_1,\ldots,Q_q$ are in $m-$subgeneral position with respect to $\mathcal{X}$ we can apply Lemma \ref{l:37} to $Q_1^{d/d_1},\ldots, Q_q^{d/d_q}$. Then, for all $m+1\leq j\leq q$, we have 
\begin{align}\label{e:60}
\dfrac{d}{d_j}\ord_\mathfrak{p}(Q_j({\bf x}))\deg\mathfrak{p}\leq d\cdot e_\mathfrak{p}({\bf x}) \deg \mathfrak{p}+c_2,
\end{align}
where 
\begin{align}\label{e:61}
c_2=\left(6\max\{(m+1)\triangle, d\}\right)^{(n+1)(N^2+N)}\left(h(F_\mathcal{X})+h(Q_1^{d/d_1},\ldots,Q_q^{d/d_q})\right).
\end{align}
 Notice that $c_2\geq 0.$ Thus,
\begin{align*}
e_\mathfrak{p}(\psi({\bf x})) \deg \mathfrak{p}&=\min\left\lbrace d\cdot e_\mathfrak{p}({\bf x}), \frac{d}{d_1}\ord_\mathfrak{p}(Q_1({\bf x})),\ldots,\frac{d}{d_q}\ord_\mathfrak{p}(Q_q({\bf x}))\right\rbrace \deg \mathfrak{p}\notag\\
&\leq d\cdot e_\mathfrak{p}({\bf x}) \deg \mathfrak{p}.
\end{align*}
Hence, for $1\leq i \leq q,$
\begin{align}\label{e:62}
\lambda_{\mathfrak{p},Y_{N_1+i}}(\psi({\bf x}))&=\left(\frac{d}{d_i}\ord_\mathfrak{p}(Q_i({\bf x}))-e_\mathfrak{p}(\psi({\bf x}))\right) \deg \mathfrak{p}\notag\\
&\geq \left(\dfrac{d}{d_i}\ord_\mathfrak{p}(Q_i({\bf x}))-d\cdot e_\mathfrak{p}({\bf x})\right) \deg \mathfrak{p}.
\end{align}
On the other hand,
\begin{align*}
\sum_{i=1}^q\dfrac{d}{d_i}\lambda_{\mathfrak{p},Q_i}({\bf x})=\sum_{i=1}^q\left(\dfrac{d}{d_i}\ord_\mathfrak{p}(Q_i({\bf x}))-d\cdot e_\mathfrak{p}({\bf x})-\dfrac{d}{d_i}e_\mathfrak{p}(Q_i)\right) \deg \mathfrak{p}
\end{align*}
is smaller than
\begin{align*}
\sum_{i=1}^q\left[\left(\dfrac{d}{d_i}\ord_\mathfrak{p}(Q_i({\bf x}))-d\cdot e_\mathfrak{p}({\bf x})\right) \deg \mathfrak{p}-c_2\right]-q\min_{1\leq i\leq q}\dfrac{d}{d_i}e_\mathfrak{p}(Q_i) \deg \mathfrak{p}+q\cdot c_2,
\end{align*}
which by \eqref{e:60} does not exceed
\begin{align*}
\sum_{i=1}^{m}\left[\left(\dfrac{d}{d_i}\ord_\mathfrak{p}(Q_i({\bf x}))-d\cdot e_\mathfrak{p}({\bf x})\right) \deg \mathfrak{p}-c_2\right]-q\min_{1\leq i\leq q}\dfrac{d}{d_i}e_\mathfrak{p}(Q_i) \deg \mathfrak{p}+q\cdot c_2.
\end{align*}
Combining with \eqref{e:62}, we have
\begin{align}\label{e:103}
\sum_{i=1}^q\dfrac{d}{d_i}\lambda_{\mathfrak{p},Q_i}({\bf x})&\leq \sum_{i=1}^m\lambda_{\mathfrak{p},Y_{N_1+i}}(\psi({\bf x}))-mc_2-q\min_{1\leq i\leq q}\dfrac{d}{d_i}e_\mathfrak{p}(Q_i) \deg \mathfrak{p}+q\cdot c_2\notag\\
&\leq \max_I\sum_{i\in I}\lambda_{\mathfrak{p},Y_{i}}(\psi({\bf x}))-q\min_{1\leq i\leq q}\dfrac{d}{d_i}e_\mathfrak{p}(Q_i) \deg \mathfrak{p}+(q-m)\cdot c_2.
\end{align}
Here the maximum is taken over all subsets $I$ of $\{N_1+1,\ldots, N_1+q\}$ with cardinality $m+1$.
Combining with \eqref{e:56}, we have
\begin{align}\label{e:63}
\sum_{\mathfrak{p}\in S}\sum_{i=1}^q\dfrac{d}{d_i}\lambda_{\mathfrak{p},Q_i}({\bf x})\leq ((m-n+1)(n+1)+\epsilon)h(\psi({\bf x}))+\tilde{b}'_{\epsilon}(h(F_\mathcal{X})+1)\notag\\
+q\cdot h(Q_1^{d/d_1},\ldots,Q_q^{d/d_q})+(q-m)|S|\cdot c_2. 
\end{align}
We may conclude the proof of the theorem by the following facts. Firstly, if $P$ is one of the homogeneous polynomials in $K[Y_0,\ldots,Y_{N_1+q}]$ defining $\mathcal{R}_\epsilon$, then $G=P\circ \psi$ is a homogeneous polynomials of degree $d\cdot \deg P$ in $K[X_0,\ldots,X_N]$ and all such $G$ form an effectively computable finite union $\mathcal{W}_\epsilon$ of proper algebraic subsets of $\mathbb{P}^N(K)$ with degree bounded above by $$2(2n+1)d^{n+1}\bigtriangleup_\mathcal{X}\left(\binom{d+N}{N}+q+1\right)\epsilon^{-1}+d$$
by \eqref{e:59}. Secondly, it is easy to check that
$$dh({\bf x})\leq h(\psi({\bf x}))\leq dh({\bf x})+h(Q_1^{d/d_1},\ldots,Q_q^{d/d_q}).$$
Hence, \eqref{e:63} becomes
\begin{align*}
\sum_{\mathfrak{p}\in S}\sum_{i=1}^q\dfrac{d}{d_i}\lambda_{\mathfrak{p},Q_i}({\bf x})\leq ((m-n+1)(n+1)+\epsilon)dh({\bf x})+\tilde{b}'_{\epsilon}(h(F_\mathcal{X})+1)\\
+(q+((m-n+1)(n+1)+\epsilon))\cdot h(Q_1^{d/d_1},\ldots,Q_q^{d/d_q})+(q-m)|S|\cdot c_2
\end{align*}
and \eqref{e:55} becomes
$$h({\bf x})\leq \frac{1}{d}\tilde{b}_\epsilon(h(F_\mathcal{X})+1).$$
Combining with \eqref{e:b} and \eqref{e:57}, \eqref{e:58} the constants $C_\epsilon, C'_\epsilon$ in the assertion can be given by
\begin{align}\label{e:c}
C_\epsilon=\frac{1}{d}a_{\epsilon/m-n+1}\cdot (d^{n+1}+(n+1)\triangle_\mathcal{X}d^nh(Q_1^{d/d_1},\ldots,Q_q^{d/d_q}))\cdot (h(F_\mathcal{X})+1)
\end{align}
and
\begin{align}\label{e:c'}
C'_\epsilon=\frac{1}{d}a'_{\epsilon/m-n+1}\cdot (m-n+1)\cdot (d^{n+1}+(n+1)\triangle_\mathcal{X}d^nh(Q_1^{d/d_1},\ldots,Q_q^{d/d_q}))\cdot (h(F_\mathcal{X})+1)\notag\\
+\frac{1}{d}\cdot(q+((m-n+1)(n+1)+\epsilon))\cdot h(Q_1^{d/d_1},\ldots,Q_q^{d/d_q})+\frac{1}{d}\cdot (q-m)|S|\cdot c_2,
\end{align}
where $a_{\epsilon/m-n+1}$ and $a'_{\epsilon/m-n+1}$ are the constants from Theorem \ref{t:41}.
\section*{Acknowledgements}
    This work was done during a stay of the author at the Vietnam Institute for Advanced Study in Mathematics (VIASM). We would like to thank the staff here, in particular the support of VIASM. This work is supported by a NAFOSTED grant of Viet Nam.

{\bf Giang Le.}
{\it Department of Mathematics, Hanoi National University of Education,}
{\it 136-Xuan Thuy, Cau Giay, Hanoi, Vietnam.}\\
\textit{E-mail: legiang01@yahoo.com}

\end{document}